\documentclass[11pt,reqno,tbtags]{amsart}
\usepackage{amssymb}
\numberwithin{equation}{section}
\usepackage{natbib}
\bibpunct[, ]{[}{]}{;}{n}{,}{,}

\newtheorem{theorem}{Theorem}[section]
\newtheorem{lemma}[theorem]{Lemma}

\newtheorem{corollary}[theorem]{Corollary}

\theoremstyle{definition}

\newtheorem*{remark}{Remark}

\newtheorem*{ack}{Acknowledgement}

\theoremstyle{remark}

\newcounter{thmenumerate}

\newcounter{xenumerate}   

\newcommand{\refT}[1]{Theorem~\ref{#1}}
\newcommand{\refC}[1]{Corollary~\ref{#1}}
\newcommand{\refL}[1]{Lemma~\ref{#1}}
\newcommand{\refS}[1]{Section~\ref{#1}}

\begingroup
  \divide\count255 by 60
  \multiply\count255 by -60
  \advance\count255 by \time
  \ifnum \count255 < 10 \xdef\klockan{\the\count1.0\the\count255}
  \else\xdef\klockan{\the\count1.\the\count255}\fi
\endgroup

\newcommand{\sumil}{\sum_{i=1}^\ell}
\newcommand{\sumjl}{\sum_{j=1}^\ell}
\newcommand{\prodil}{\prod_{i=1}^\ell}

\newcommand{\suml}{\sum_{\ell=0}^\infty}

\newcommand{\summi}{\sum_{m=1}^\infty}

\newcommand\set[1]{\ensuremath{\{#1\}}}
\newcommand\bigset[1]{\ensuremath{\bigl\{#1\bigr\}}}

\newcommand\xpar[1]{(#1)}
\newcommand\bigpar[1]{\bigl(#1\bigr)}
\newcommand\Bigpar[1]{\Bigl(#1\Bigr)}

\newcommand\lrpar[1]{\left(#1\right)}

\newcommand\xcpar[1]{\{#1\}}

\newcommand\biggabs[1]{\biggl|#1\biggr|}

\def\rompar(#1){\textup(#1\textup)}    

\def\xexp(#1){e^{#1}}

\newcommand\ntoo{\ensuremath{{n\to\infty}}}

\newcommand\bmin{\wedge}

\newcommand\ie{i.e.\spacefactor=1000}
\newcommand\eg{e.g.\spacefactor=1000}
\newcommand\viz{viz.\spacefactor=1000}

\newcommand\ii{\mathrm{i}}

\newcommand{\tend}{\longrightarrow}
\newcommand\dto{\overset{\mathrm{d}}{\tend}}

\newcommand\eqd{\overset{\mathrm{d}}{=}}

\newcommand\bbR{\mathbb R}
\newcommand\bbC{\mathbb C}

\newcommand\bbZ{\mathbb Z}

\newcounter{CC}
\newcommand{\CC}{\stepcounter{CC}\CCx} 
\newcommand{\CCx}{C_{\arabic{CC}}}     
\newcommand{\CCdef}[1]{\xdef#1{\CCx}}     
\newcounter{cc}
\newcommand{\cc}{\stepcounter{cc}\ccx} 
\newcommand{\ccx}{c_{\arabic{cc}}}     
\newcommand{\ccdef}[1]{\xdef#1{\ccx}}     

\newcommand\E{\operatorname{\mathbb E{}}}
\renewcommand\P{\operatorname{\mathbb P{}}}
\newcommand\Var{\operatorname{Var}}

\newcommand\Ge{\operatorname{Ge}}

\newcommand\gb{\beta}
\newcommand\gd{\delta}
\newcommand\gD{\Delta}
\newcommand\gf{\varphi}
\newcommand\gam{\gamma}

\newcommand\gs{\sigma}
\newcommand\gss{\sigma^2}
\newcommand\eps{\varepsilon}

\newcommand\cT{{\mathcal T}}

\def\[#1]{[\![#1]\!]}

\renewcommand{\=}{:=}

\newcommand\lhs{left hand side}
\newcommand\rhs{right hand side}

\newcommand{\GW}{Galton--Watson}
\newcommand{\GWp}{\GW{} process}
\newcommand{\GWt}{\GW{} tree}
\newcommand{\cGWt}{conditioned \GW{} tree}

\newcommand{\tn}{T_n}
\newcommand{\tnx}{\tn^*}
\newcommand{\hT}{\widehat T}
\newcommand{\htn}{\hT_n}
\newcommand{\fn}{f_n}
\newcommand{\hn}{h_n}
\newcommand{\vw}{v\bmin w}
\newcommand{\gdxx}[2]{\gD(#1,#2)}
\newcommand{\gdbd}{\gdxx{\gb}{\gd}}
\newcommand{\gdbdi}{\gdxx{\gb_1}{\gd_1}}
\newcommand{\gdbdii}{\gdxx{\gb_1}{\gd_2}}
\newcommand{\gdbr}{\gdxx{\gb}{\rho}}
\newcommand{\ogdbd}{\overline{\gdbd}}
\newcommand{\ogdbdx}{\ogdbd\setminus\set{1}}
\newcommand{\gdbde}{\gdbd\cap\be}
\newcommand{\sumvtn}{\sum_{v\in\tn}}
\newcommand{\sumvwtn}{\sum_{v,w\in\tn}}
\newcommand{\sumvt}{\sum_{v\in\cT}}
\newcommand{\sumvwt}{\sum_{v,w\in\cT}}
\newcommand{\gsy}{\gs_\eta}
\newcommand{\gssy}{\gss_\eta}
\newcommand{\etae}{\eta_e}
\newcommand{\oo}{o}
\newcommand{\lx}[1]{L_{#1}}
\newcommand{\lv}{\lx{v}}
\newcommand\cor{C_0(\bbR)}
\newcommand{\bXn}{X_n}
\newcommand{\xn}{\bXn}
\newcommand{\doo}{d_\oo}
\newcommand{\phifz}{\Phi\xpar{F(z)}}
\newcommand{\yyxx}[2]{\widetilde P_{#1,#2}}
\newcommand{\yylm}{\yyxx{\ell}m}
\newcommand{\oqq}{o\bigpar{|z-1|^{1/2}}}
\newcommand{\be}{B(1,\eps)}
\newcommand{\br}{B_\rho}
\newcommand{\ba}{\overline{A}}
\newcommand{\ce}{C_\eps}
\newcommand{\psint}{\Psi(n,t)}
\newcommand{\hk}{^{(k)}}
\newcommand{\tauq}{T}

\newcommand\ise{_{\text{\sc ise}}} 

\newcommand\fise{f\ise}

\newcommand\qq{^{1/2}}
\newcommand\qqw{^{-1/2}}
\newcommand\qw{^{-1}}
\newcommand\qww{^{-2}}
\newcommand\qwwi{^{-3/4}}
\newcommand\qqqq{^{1/4}}
\newcommand\qc{^{3/2}}
\newcommand\qcw{^{-3/2}}

\newcommand\wz{Z}

\newcommand\PROB{\P}
\newcommand\EXP{\E}
\newcommand\IND[1]{\boldsymbol1_{[#1]}}

\newcommand\sa{\Sigma_1(n_1,n_2)}

\newcommand\sbm{\Sigma_2(m)}
\newcommand\sbx{\Sigma_2(n-1)}
\newcommand\scl{\Sigma_3(\ell)}
\newcommand\scx[1]{\Sigma_3(#1)}

\newcommand\gfeta{\gf_\eta}

\newcommand\REM[1]{{\raggedright\texttt{[#1]}\par\marginal{XXX}}}

\newcommand\urladdrx[1]{\urladdr{\def~{\~{}}#1}}

\begin{document}
\title[Distances between pairs of vertices and vertical profile]
{Distances between pairs of vertices and vertical profile
in conditioned Galton--Watson trees}

\date{December 17, 2008} 

\author{Luc Devroye}
\address{School of Computer Science, McGill University, 3480
  University Street, Montreal, Canada H3A 2K6}
\email{lucdevroye@gmail.com} 
\urladdrx{http://cg.scs.carleton.ca/~luc/}

\author{Svante Janson}
\address{Department of Mathematics, Uppsala University, PO Box 480,
SE-751~06 Uppsala, Sweden}
\email{svante.janson@math.uu.se}
\urladdrx{http://www.math.uu.se/~svante/}

\thanks{\emph{CR Categories.} 3.74, 5.25, 5.5.} 

\thanks{The first author's research was sponsored by NSERC Grant A3456 and FQRNT Grant  90-ER-0291}

\begin{abstract} 
We consider a \cGWt{} and
prove an estimate of the number of pairs of vertices 
with a given distance, or, equivalently,
the number of paths of a given length.

We give two proofs of this result, one probabilistic and the other
using generating functions and singularity analysis.

Moreover,
the second proof yields a more general estimate for generating
functions, which is used to prove a conjecture by
Bousquet--M\'elou and Janson \cite{SJ185}, 
saying that the vertical
profile of a randomly labelled \cGWt{} converges in distribution, 
after suitable normalization, 
to the density of ISE (Integrated Superbrownian Excursion).
\end{abstract}

\maketitle

\section{Introduction and results}\label{S:intro}

Let $\tn$ be a \cGWt, \ie, the random rooted tree $\cT$ obtained as the
family tree of a \GWp{} with some given offspring distribution $\xi$,
conditioned on the number of vertices $|\cT|=n$. We will always assume
that 
\begin{align}\label{Axi}
\E\xi=1 \qquad \text{and} \qquad
0<\gss\=\Var\xi<\infty
\end{align}
In other words, the \GWp{} is critical and with finite variance, and
$\P(\xi=1)<1$. 
(Note that this entails $0<\P(\xi=0)<1$.)
It is well-known that this assumption is without essential loss of
generality, and that the resulting random trees are essentially the
same as the simply generated families of trees introduced by Meir and
Moon \cite{MM}.
The importance of this construction 
lies in that 
many combinatorially interesting random trees are of this type,
for example 
random plane (= ordered) trees, 
random unordered labelled trees (Cayley trees), 
random binary trees, 
and (more generally) random $d$-ary trees. 
For further examples see \eg{}
Aldous \cite{AldousII} and Devroye \cite{Devroye}.

We consider 
only $n$ such that $\tn$ exists, \ie,
such that $\P(|\cT|=n)>0$.
The \emph{span} of $\xi$ is defined to be the largest integer $d$ such
that $\xi\in d\bbZ$ a.s. If the span of $\xi$ is $d$, then $\tn$ exists
only for $n\equiv 1 \pmod d$, and it exists for all large such $n$.

We consider in this paper two types of properties of $\tn$ that turn
out to have proofs using a common argument.
First, for an arbitrary rooted tree $\tau$, 
let $P_k(\tau)$, $k\ge1$, be the number of (unordered) pairs of
vertices \set{v,w} in $\tau$ such that the distance $d(v,w)=k$;
equivalently, $P_k(\tau)$ is the number of paths of length $k$ in $\tau$.
Our first result is an estimate, uniform in all $k$ and $n$, 
of the expectation of this number $P_k(\tn)$ for a conditioned
Galton--Watson tree $\tn$.

We let in this paper $C_1,C_2$ and $c_1,c_2$ denote various positive constants
that may depend on (the distribution of) $\xi$, and sometimes later
$\eta$ introduced below, but not on $n$, $k$ and other variables
unless explicitly stated. 
Recall that we tacitly assume \eqref{Axi}. 

\begin{theorem}
  \label{T1}
There exists a constant $\CC$ 
such that for all 
$k\ge1$ and $n\ge1$,
$\E P_k(\tn)\le \CCx nk$.
\end{theorem}

One way to interpret this result is that the expected number of
vertices of distance $k$ from a randomly chosen vertex in $\tn$ is
$O(k)$. In other words, if $\tnx$ is $\tn$ randomly rerooted, and
$\wz_k(\tau)$ is the number of vertices of distance $k$ from the root in
a rooted tree $\tau$, then the following holds.

\begin{corollary}
  \label{C1}
$\E \wz_k(\tnx)=O(k)$, uniformly in all $k\ge1$ and $n\ge1$.
\end{corollary}

This can be compared to \cite[Theorem 1.13]{SJ167}, which shows that
\begin{equation}
\label{wk}  
\E \wz_k(\tn)=O(k), 
\end{equation}
again uniformly in $k$ and $n$.
Note that in the special case when $\tn$ is a random (unordered)
labelled tree, $\tnx$ has the same distribution, so \refC{C1} reduces
to \eqref{wk}. However, in general, a randomly rerooted \cGWt{} is not
a \cGWt. 

\begin{remark}
  The emphasis is on uniformity in both $k$ and $n$. If we, on the
  contrary, fix $k$ and consider limits as \ntoo, we have 
$\E \wz_k(\tn)\to1+k\gss$, see Meir and Moon \cite{MM} and Janson
  \cite{SJ167,SJ188}. It is shown in \cite{SJ188} that the sequence
  $\E \wz_k(\tn)$ is not always monotone in $n$.
\end{remark}

We give a probabilistic proof of \refT{T1}, and thus of \refC{C1} too, 
in \refS{Spf1}. 

We also give a second proof by first proving 
a corresponding estimate for the generating function.
(We present two different proofs, since we find both
methods interesting, and both methods yield 
as intermediary steps in the proofs
other results that we find interesting.)
Let $\fn(z)$ be the generating function defined by
\begin{equation*}
  \fn(z)
\=\sum_{k=1}^\infty \E P_k(\tn)\,z^k
.
\end{equation*}

We will use standard singularity analysis, see \eg{} 
Flajolet and Sedgewick \cite{FS},
and define the domain, for $0<\gb<\pi/2$ and $\gd>0$,
\begin{equation*}
  \gdbd\=\set{z\in\bbC:|z|<1+\gd,\, z\neq1,\, |\arg(z-1)|>\pi/2-\gb}.
\end{equation*}
Note that $|\arg(z-1)|>\pi/2-\gb$ is equivalent to 
$|\arg(1-z)|<\pi/2+\gb$.

\begin{theorem}
  \label{Tgen1}
For every $\xi$ there exist positive constants $\CC$, $\gb$, $\gd$ such
that for all $n\ge1$, $\fn$ extends to an analytic function in $\gdbd$
with
\begin{equation}
  \label{fn}
  |\fn(z)| \le \CCx n|1-z|\qww,
\qquad z\in\gdbd.
\end{equation}
\end{theorem}

By standard singularity analysis (\ie, estimate of the Taylor
coefficients of $f_n(z)$ using Cauchy's formula and a suitable contour
in $\gdbd$), \eqref{fn} implies $\E P_k(T_n)=O(n k)$,
see Flajolet and Sedgewick \cite{FS},
Theorem VI.3 and (for the uniformity in $n$) Lemma IX.2 (applied
  to the family \set{f_n(z)/n}).
Hence, \refT{T1} follows from \refT{Tgen1}.

For each pair of
vertices $v,w$ in a rooted tree,  
the path from $v$ to $w$ consists of two (possibly empty) parts, one
going from $v$ towards the root, ending at the last common ancestor
$\vw$ of $v$ and $w$, and another part going from $\vw$ to $w$ in the
direction away from the root.
We will also prove extensions of the results above for $\tn$, where we 
consider separately the lengths of these two parts.
Define
the corresponding bivariate generating function
(now considering ordered pairs $v,w$)
\begin{equation}
  \label{hn}
\hn(x,y)
\=
\E\sumvwtn x^{d(v,\vw)} y^{d(w,\vw)}.
\end{equation}

\begin{theorem}
  \label{Tgen2}
For every $\xi$ there exist positive constants $\CC\CCdef\CCgenii$, 
$\gb$, $\gd$ such
that for all $n\ge1$, 
\begin{equation*}
  |\hn(x,y)| \le \CCx n|1-x|\qw|1-y|\qw,
\qquad x,y\in\gdbd.
\end{equation*}
\end{theorem}

Note that, by \eqref{hn} and \eqref{fn}
\begin{equation*}
  \hn(z,z)=\E \sumvwtn z^{d(v,w)} = n+2\fn(z).
\end{equation*}
Hence \refT{Tgen1} follows from \refT{Tgen2}.
We prove \refT{Tgen2} in \refS{Sgen2}.

If we define $\yylm(\tau)\=\#\set{(v,w)\in \tau:d(v,\vw)=\ell,\,d(w,\vw)=m}$,
then
singularity analysis as above (but twice) shows that \refT{Tgen2}
implies the following.
(Since $P_k=\frac12\sum_{\ell=0}^k \yyxx{\ell}{k-\ell}$, this too
implies \refT{T1}.) 

\begin{theorem}
  \label{T11}
There exists a constant $\CC$ 
such that for all 
$\ell,m\ge0$ and $n\ge1$,
$\E \yylm(\tn)\le \CCx n$.
\end{theorem}

One motivation for these results is that they (more precisely,
\refT{Tgen2}) are used to prove the second type of result in this
paper.
For this, we assume that we are given a further random variable
$\eta$. Given a rooted tree $\tau$, 
we take an independent copy $\etae$ of $\eta$ for
every edge $e\in \tau$.
We give each vertex $v$ the label $\lv$ obtained by summing $\etae$
for all $e$ in the path from the root $\oo$ to $v$. (Thus,
$\lx{\oo}=0$.)
We assume that 
\begin{align}\label{Aeta}
\E\eta=0 \qquad \text{and} \qquad 0<\gssy\=\Var\eta<\infty.
\end{align}
We further assume that $\eta$ is integer valued and with
span $1$; thus 
all labels are integers, and all integers are possible labels.

We let $X(j;\tau)$ be the number of vertices in $\tau$ with
label $j$; the sequence $(X(j;\tau))_{j=-\infty}^{\infty}$ 
is the \emph{vertical profile} of the labelled tree.

For the random tree $\tn$, we assume that the variables $\etae$ are
independent of $\tn$. The vertical profile $X(j;\tn)$ then is a random function
defined for $j\in\bbZ$; we write $\xn(j)\=X(j;\tn)$ and extend the
domain of $\xn$ to $\bbR$ by linear interpolation between the integer points.
Our next theorem says that this function $\xn$, suitable
normalized, converges in distribution in the space 
$\cor$ of continuous functions on
$\bbR$ that tend to 0 at $\pm\infty$; we equip $\cor$ with the usual
uniform topology defined by the supremum norm.
Let, further, $\fise$ denote the density of the random measure ISE
introduced by Aldous \cite{AldousISE}; $\fise$ is a random continuous
function
with (random) compact support, see 
Bousquet--M\'elou and Janson \cite[Theorem 2.1]{SJ185}.

\begin{theorem}\label{T2}
With the assumptions above,
including \eqref{Axi} and \eqref{Aeta}, 
let $\gam\=\gsy\qw\gs\qq$.
Then, as \ntoo,
  \begin{equation}\label{t2a}
	\frac1n \gam\qw n\qqqq \xn\bigpar{\gam\qw n\qqqq x}
\dto \fise(x),
  \end{equation}
in the space $\cor$ with the usual uniform topology.
Equivalently,
\begin{equation}\label{t2b}
 n\qwwi \xn\bigpar{n\qqqq x}
\dto \gam\fise(\gam x).
\end{equation}
\end{theorem}

Note that the random functions on the left and \rhs{s} of \eqref{t2a}
and \eqref{t2b} 
are density functions, \ie,
non-negative functions with integral 1.

\begin{corollary}\label{C2}
If \ntoo{} and $j_n/n\qqqq \to x$, where
$-\infty<x<\infty$,
then 
$n\qwwi X(j_n;\tn)\dto \gam \fise(\gam x)$.
\end{corollary}

The limit law is characterized in \cite{BM} by a formula for its
Laplace transform.

\refT{T2} was conjectured in \cite{SJ185}, and proved there in two
special cases, \viz{} when $\xi$ has the Geometric distribution
$\Ge(1/2)$ and thus $\tn$ is a 
random ordered tree, and $\eta$ is uniformly distributed on either
\set{-1,1} or \set{-1,0,1}.
Moreover, it was shown there \cite[Remark 3.7]{SJ185} that the proof
given in \cite{SJ185} applies generally under the assumptions above,
provided the following estimate holds.

\begin{lemma}
  \label{L0}
Under the assumptions above,
there exists a constant $\CC$ such that for all $n\ge1$ and
$t\in[-\pi,\pi]$,
\begin{equation}
  \label{l2}
\E\biggabs{\frac1n\sum_j X(j;\tn)e^{\ii jt}}^2
\le \frac{\CCx}{1+nt^4}.
\end{equation}
\end{lemma}

We prove \refL{L0}, and thus \refT{T2}, in \refS{Sl2},
assuming \refT{Tgen2}.
Finally, we prove \refT{Tgen2}, using singularity analysis again, in
\refS{Sgen2}, which completes the proof of all other results.

\begin{ack}
  This research was mainly done during a workshop at Bellairs Research
  Institute in Barbados, March 2006,
and completed during a visit of SJ to 
Centre de recherches math\'ematiques, Universit\'e de
Montr\'eal, October 2008.
\end{ack}

\section{First proof of \refT{T1}} \label{Spf1}

In a rooted tree $\tau$, let
$Q_k(\tau)$, $k\ge1$, denote the number of (unordered) pairs of vertices
at path distance $k$ from each other such that the path
between them visits the root,
and let $Q'_k(\tau)$ be the number of such pairs where 
the root cannot be one of the two vertices in the pair;
thus $Q_k(\tau)=Q'_k(\tau) + Z_k(\tau)$.
Then, in the Galton--Watson tree $\cT$, if $\xi$ is the number of
children of the root, and the subtrees rooted at these children are
denoted $\cT_1,\dots,\cT_\xi$,
\begin{equation}
  \label{qkt}
Q'_k(\cT) = \sum_{(r,s):  1 \le r < s \le \xi} \sum_{j=0}^{k-2} Z_j  (\cT_r)
Z_{k-2-j} (\cT_s) 
\end{equation}
and thus, since we assume $\cT$ to be critical, \ie, $\E\xi=1$, so
$\E Z_k(\cT)=1$ for every $k$,
\begin{equation}
  \label{qk}
\E Q_k(\cT)=\E Z_k(\cT)+\E Q'_k(\cT) = 1+ \E\frac{\xi(\xi-1)}2 (k-1)
=1+(k-1)\frac{\gss}2.
\end{equation}

Let $\htn$ denote the random subtree of $T_n$ 
rooted at a uniformly selected random vertex. 
(Note the difference from $\tnx$ in \refC{C1}; in $\tnx$ we keep
all $n$ vertices, but in $\htn$ we keep only the vertices below the
new root.)
Then, clearly,
$$
\EXP \{ P_{k}(T_n) \} = n \EXP \{ Q_k  (\hT_n) \}.
$$
Consequently, \refT{T1} is equivalent to:
\begin{theorem}
  \label{TP1}
There exists a constant $\CC$ 
such that for all 
$k\ge1$ and $n\ge1$,
$\E Q_k(\htn)\le \CCx k$.
\end{theorem}

In order to prove this, we will need a related, but different,
estimate for the \cGWt{} $T_n$.

\begin{theorem}
  \label{TQ}
There exists a constant $\CC\CCdef\CCTQ$ 
such that for all 
$k\ge1$ and $n\ge1$,
$\E Q_k(\tn)\le \CCx k\sqrt n$.
\end{theorem}

It is easy to see $\E Q_k(\tn)\ge \cc n^{3/2}$ when $k\sim\sqrt n$,
so the estimate in \refT{TQ} then is of the right order; in particular, the
estimate in \refT{TP1} for $\htn$ does \emph{not} hold for $\tn$.

To prove these theorems we use a few more or less standard estimates.

\begin{lemma} 
  \label{L2}
Assume, as above, \eqref{Axi}, 
and let $d$ be the span of $\xi$.
Let $S_n\=\sum_{i=1}^n\xi_i$, where $\xi_i$ are independent copies of $\xi$.
Then, for $n\equiv 1\pmod d$, 
  \begin{equation}\label{tail}
\P(|\cT|=n)=\frac1n\P\xpar{S_n=n-1}	
\sim \frac{d}{\sigma \sqrt{2 \pi}\, n^{3/2}}
\quad\text{as }\ntoo.	
  \end{equation}

More generally, let $W_\ell\=\sum_{i=1}^\ell|\cT_i|$ be the size of the
union of $\ell$ independent copies of $\cT$, or equivalently, the
total progeny of a Galton--Watson process started with $\ell$
individuals,
with offspring distribution $\xi$.
Then, for all $\ell\ge1$ and
$n\ge1$,  
\begin{equation}\label{l2b}
  \P(W_\ell=n)=\frac{\ell}n\P(S_{n}=n-\ell)	
\le \CC \ell n\qcw\exp(-\cc\ell^2/n).
\CCdef\CClb \ccdef\cclb
\end{equation}
In particular,
\begin{equation}\label{l2c}
  \P(W_\ell=n)\le \CC n\qw.
\end{equation}
\end{lemma}

\begin{proof}
The identity in \eqref{l2b} is well-known, see \eg, Dwass
\cite{Dwass}, Kolchin \cite[Lemma 2.1.3, p.~105]{Kolchin} and
\citet{Pitman:enum}. 
The identity in \eqref{tail} is the special case $\ell=1$, and
the well-known tail estimate in \eqref{tail} then follows by the local
central limit theorem, see, \eg, Kolchin \cite[Lemma 2.1.4, p.~105]{Kolchin}.

Similarly, the inequality in \eqref{l2b} follows by the 
estimate
$\P(S_n=n-\ell)\le \CClb n\qqw \exp(-\cclb\ell^2/n)$
from \cite[Lemma 2.1]{SJ167}.
The inequality $e^{-x}\le x\qqw$ yields \eqref{l2c}.
\end{proof}

\begin{lemma}  
  \label{L1a}
For every $r>0$ there is a constant $\CC(r)\CCdef\CCr$ such that
for all $k\ge0$ and $n\ge1$, $\E Z_k(T_n)^r\le \CCr(r) n^{r/2}$.
\end{lemma}

\begin{proof}
For any rooted tree $\tauq$, let $\tauq\hk$ be the tree pruned at height
$k$, i.e., the subtree consisting of all vertices of distance at most $k$
from the root. 
Let $\tau$ be a given rooted tree of height $k$, and let
$m\=Z_k(\tau)$, the number of leaves at maximal depth.
Note that if $\tau= \tauq\hk$ for some tree $\tauq$, then $|\tauq|=n$
if and only if $\tauq$ has $n-|\tau|$ vertices at greater depth than $k$,
and thus $N\=n-|\tau|+m$ vertices at depth $k$ or greater.
Hence, with $W_m$ as in \refL{L2} and using \eqref{tail} and
\eqref{l2b}, for any $r>0$ and assuming $N>0$ (otherwise the
probability is 0),
\begin{equation*}
  \begin{split}
\P\bigpar{T_n\hk=\tau}
&=
\frac{\P\bigpar{\cT\hk=\tau,\,|\cT|=n}}{\P(|\cT|=n)}
=\frac{\P(\cT\hk=\tau)\P(W_m=N)}{\P(|\cT|=n)}
\\& 
\le \CC n\qc\P(\cT\hk=\tau)m N\qcw e^{-\cclb m^2/N} 
\\& 
\le \CC(r) n\qc\P(\cT\hk=\tau)m N\qcw (m^2/N)^{-r/2} 
\\& 
= \CCx(r) n\qc N^{r/2-3/2} m^{1-r}\P(\cT\hk=\tau).
  \end{split}
\end{equation*}
If $r\ge3$, this yields, since $N\le n$, the estimate
\begin{equation*}
\P(T_n\hk=\tau)\le \CCx(r) n^{r/2} m^{1-r}\P(\cT\hk=\tau),   
\end{equation*}
and
summing over all $\tau$ of height $k$ with $Z_k(\tau)=m$ we obtain
\begin{equation*}
\P(Z_k(T_n)=m)\le \CCx(r) n^{r/2} m^{1-r}\P(Z_k(\cT)=m).
\end{equation*}
Consequently,
\begin{equation*}
  \begin{split}
\E Z_k(T_n)^r &=\summi m^r \P(Z_k(T_n)=m)
\\&
\le \CCx(r) n^{r/2} \summi m\P(Z_k(\cT)=m)
\\&
= \CCx(r) n^{r/2} \E Z_k(\cT)
= \CCx(r) n^{r/2}.	
  \end{split}
\end{equation*}
This proves the result for $r\ge3$, and the result for $0<r<3$ follows
by Lyapounov's (or H\"older's) inequality.
\end{proof}

\begin{lemma}
  \label{L1b}
For all $k\ge1$ and $n\ge1$, $\E Z_k(T_n)\le \CC(k\bmin\sqrt n)
\CCdef\CCLi$.
Equivalently, 
for all $k\ge0$ and $n\ge1$, $\E Z_k(T_n)\le \CC((k+1)\bmin\sqrt n)
\CCdef\CCLia$.
\end{lemma}

\begin{proof}
  The estimate $\E Z_k(T_n)\le \CC k$ is \eqref{wk}, which is proved in 
\cite[Theorem 1.13]{SJ167}. 
The estimate $\E Z_k(T_n)\le \CC\sqrt n$ is proved in \refL{L1a}.
\end{proof}

\begin{remark} 
The estimate $\E Z_k(T_n)\le \CC\sqrt n$ was proved by
Drmota and Gittenberger \cite{DrmotaG:width}, assuming that $\xi$   
has an exponential moment; 
in fact, they then
prove the stronger bound 
$\E Z_k(T_n)\le \CC\sqrt n \exp(-\cc k/\sqrt n)$.
The bound in \refL{L1b} can be further improved to 
$\E Z_k(T_n)\le \CC k \exp(-\cc k^2/ n)$, but we do not know a reference for
this estimate. (Details may appear elsewhere.)  
\end{remark}

\begin{remark} 
  Note that \refL{L1a} yields an estimate $O(n^{r/2})$ of the $r$th moment of
  $Z_k(T_n)$ for an arbitrary $r$ assuming only a second moment of $\xi$.
This is in contrast to the estimate \eqref{wk}, where the
  corresponding estimate $\E Z_k(T_n)^r = O(k^r)$ is valid 
(for integer $r\ge1$ at least) 
if $\xi$ has  a finite $r+1$:th moment, but not otherwise (not even
  for a fixed $k\ge2$); 
one direction is by Theorem 1.13 in \cite{SJ167}, and the converse
  follows from the discussion after Lemma  2.1 in \cite{SJ167}.
\end{remark}

\begin{proof}[Proof of \refT{TQ}]
We have 
$Q_k (T_n)=Q'_k (T_n)+Z_k (T_n)$, and $\E Z_k (T_n)\le \CCLi k$ by
\refL{L1b},
so it suffices to show that $\E Q'_k (T_n)\le \CC k\sqrt n$.

We use \eqref{qkt}, condition on $|\cT|=n$ and take
expectations. Using the symmetry and recalling that
$\cT_1,\dots,\cT_\xi$ are independent and 
$(\cT_i\mid |\cT_i|=n_i)\eqd(\cT\mid |\cT|=n_i)\eqd T_{n_i}$ for any
$n_i$,
we obtain,
with $p_\ell\=\P(\xi=\ell)$ and $q_m\=\P(|\cT|=m)$,
\begin{equation*}
  \begin{split}
\EXP \left\{ Q'_k (T_n) \right\}
&=
\frac{ \EXP \left\{ \IND{\xi \ge 2, |\cT|=n } \sum_{1 \le r<s \le \xi}
  \sum_{j=0}^{k-2} Z_{j} (\cT_r) Z_{k-2-j} (\cT_s)  \right\}} 
{ \PROB \{ |\cT| = n \} } \\
&=
\frac{ \EXP \left\{ \IND{|\cT|=n }\binom{\xi}2 \sum_{j=0}^{k-2} Z_{j}
  (\cT_1) Z_{k-2-j} (\cT_2)  \right\}} 
{\PROB \{ |\cT| = n \} } \\
&=
q_n\qw \sum_{\ell=2}^\infty p_\ell \binom{\ell}2 
           \sum_{n_1,n_2\ge1}
	   q_{n_1} q_{n_2}
\P\lrpar{\sum_{i=3}^\ell|\cT_i|=n-1-n_1-n_2}
\\&\qquad\qquad\times
	   \sum_{j=0}^{k-2}
	   \EXP \left\{ Z_{j}(T_{n_1}) \right\}
	   \EXP \left\{ Z_{k-2-j}(T_{n_2})  \right\}.
  \end{split}
\end{equation*}
We begin with the inner sum over $j$, $\sa$ say.
By symmetry, we consider only $n_1\le n_2$, and then we obtain from
\refL{L1b} the estimates
$\E Z_{k-2-j}(T_{n_2})\le \CCLia((k-1-j)\bmin n_2\qq)\le \CCLia(k\bmin n_2\qq)$
and
\begin{equation*}
  \sum_{j=0}^{k-2}  \EXP \left\{ Z_{j}(T_{n_1}) \right\}
\le 
\begin{cases}
  \sum_{j=0}^{k-2}  \CCLia(j+1)
\le
\CCLia k^2
, 
\\
\E\sum_{j=0}^{\infty} Z_{j}(T_{n_1}) =n_1. 
\end{cases}
\end{equation*}
Hence
\begin{equation}\label{sa}
  \sa \le \CC(k^2\bmin n_1)(k\bmin n_2\qq).
\end{equation}

Let $m\=n_1+n_2$ and sum over $n_1,n_2$ with a given sum $m$. We have
by \eqref{sa} and \eqref{tail},
\begin{equation}\label{sbm}
  \begin{split}
\sbm&:=  \sum_{n_1+n_2=m} q_{n_1}q_{n_2}\sa
\\&
\le
2\sum_{n_1=1}^{m/2} q_{n_1}q_{m-n_1}\CCx(k^2\bmin n_1)(k\bmin (m-n_1)\qq)
\\
&\le
\CC\sum_{n_1=1}^{m/2} n_1\qcw (m-n_1)\qcw(k^2\bmin n_1)
  (k\bmin (m-n_1)\qq)
\\
&\le
\CC\frac{k\bmin m\qq}{m\qc} \sum_{n_1=1}^{m/2} \frac{k^2\bmin n_1}{n_1\qc}
\\
&\le
\CC\frac{k\bmin m\qq}{m\qc} (k\bmin m\qq)
=
\CCx\frac{k^2\bmin m}{m\qc}.
  \end{split}
\end{equation}
We define further
\begin{equation}\label{scl}
\scl\=\sum_{m=2}^{n-1} \sbm \P\Bigpar{\sum_{i=3}^\ell|\cT_i|=n-1-m};
\end{equation}
thus
\begin{equation}\label{qg}
  \E Q_k'(T_n) = 
q_n\qw \sum_{\ell=2}^\infty p_\ell \binom{\ell}2 \scl
\le
\CC n\qc \sum_{\ell=2}^\infty p_\ell {\ell}^2 \scl.
\end{equation}
We will show that $\scl\le \CC k/n$, uniformly in $\ell\ge2$, and the result
follows by \eqref{qg}, recalling that $\sum_\ell p_\ell\ell^2=\E\xi^2<\infty$.
(The proof can be simplified in the case $\E\xi^3<\infty$, when it
suffices to show that $\scl\le \CC k\ell/n$.)

First, if $\ell=2$, the only non-zero term in \eqref{scl} is for
$m=n-1$, which yields, by \eqref{sbm},
\begin{equation*}
  \scx2=\sbx 
\le 
\CC\frac{k^2\bmin n}{n\qc}
\le 
\CCx\frac{k\sqrt n}{n\qc}.
\end{equation*}

For $\ell>2$, we split the sum in \eqref{scl} into two parts, with
$m\le n/2$ and $m>n/2$. We have, by \eqref{sbm},  
\begin{equation*}
  \begin{split}
  \sum_{m=n/2}^{n-1} 
\sbm \P\Bigpar{\sum_{i=3}^\ell|\cT_i|=n-1-m}
&\le
\CC  \frac{k^2\bmin n}{n\qc} \P\Bigpar{\sum_{i=3}^\ell|\cT_i|\le n/2}
\\
\le
\CCx  \frac{k^2\bmin n}{n\qc} 
\le \CCx \frac k n.
  \end{split}
\end{equation*}
Similarly, using
\eqref{sbm} and \eqref{l2c} (with $\ell$ replaced by $\ell-2$),
\begin{equation*}
  \begin{split}
  \sum_{m=1}^{n/2} 
\sbm \P\Bigpar{\sum_{i=3}^\ell|\cT_i|=n-1-m}
&\le
\CC   \sum_{m=1}^{n/2}  \frac{k^2\bmin m}{m\qc}\cdot \frac 1n
\\
&\le
\frac \CCx n   \sum_{m=1}^{\infty}  \frac{k^2\bmin m}{m\qc} 
\le \CC \frac k n.
  \end{split}
\end{equation*}
Thus $\scl\le \CC k/n$, and the theorem follows by \eqref{qg}.
\end{proof}

\begin{proof}[Proof of Theorems \ref{TP1} and \refT{T1}]

Aldous \cite{Aldous:fringe} has studied the behavior of a random subtree
$\htn$ in a conditional Galton--Watson tree $\tn$.
In particular, he has the following identity
\cite[p.~242]{Aldous:fringe}, for any fixed ordered tree $\tau$ of
order at most $n$ (provided that the probabilities 
in the denominators are nonzero):
$$
\frac{\PROB \{ \hT_n = \tau \} }{ \PROB \{ \cT = \tau \} }
=
\frac{(n-|\tau|+1) \PROB \{ |\cT| = n - |\tau| +1 \} } 
{ n \PROB \{ |\cT| = n \} }
\cdot
\frac{\gamma }{ p_0},
$$
where $\gamma$ is the expected proportion of leaves in $\hT_{n-|\tau|+1}$
and $p_0 = \PROB \{ \xi = 0 \}$.
We will simply bound $\gamma \le 1$, but it is well-known that as
$n-|\tau|+1 \to \infty$, $\gamma \to p_0$,
see \eg{} Kolchin \cite[Theorem 2.3.1, p.~113]{Kolchin}.
Thus, using the well-known tail estimate \eqref{tail},
for all (permitted) $n$ and $\tau$
\begin{equation*}
  \begin{split}
\frac{\PROB \{ \hT_n = \tau \} }{ \PROB \{ \cT = \tau \} }
&\le  \CC
\frac{(n-|\tau|+1) \PROB \{ |\cT| = n - |\tau| +1 \} }
 { n \PROB\{|\cT| = n \} } 
\le \CC\CCdef\CCb  \sqrt\frac{ n }{ n-|\tau|+1} . 
  \end{split}
\end{equation*}
Hence, 
\begin{equation*}
  \begin{split}
\EXP \{ Q_k (\hT_n) \} 
&= \sum_{\tau} Q_k (\tau) \PROB \{ \hT_n = \tau \} \\
&\le \CCb\sum_{\tau} 
Q_k (\tau) \sqrt\frac{ n }{ n-|\tau|+1} \PROB \{ \cT = \tau \}\\ 
&=
\CCb\E \lrpar{Q_k (\cT) \sqrt\frac{ n }{ n-|\cT|+1}\,} \\
&\le \CC \EXP \{ Q_k (\cT) \} + \CCb \sum_{n \ge \ell > n/2} \sqrt\frac{ n }{
  n-\ell+1} \, \EXP \left\{ \IND{|\cT|=\ell} Q_k (\cT) \right\}.
  \end{split}
\end{equation*}
We have $\E\xcpar{Q_k(\cT)}\le\CC k\CCdef\CCk$ by \eqref{qk}, and, 
using \eqref{tail} and \refT{TQ},
\begin{equation*}
  \EXP \left\{ \IND{|\cT|=\ell} Q_k (\cT) \right\}
=\P\xcpar{|\cT|=\ell}  \EXP \left\{  Q_k (T_\ell) \right\}
\le \CC \ell\qcw k\ell\qq 
=\CCx k/\ell.
\end{equation*}
Consequently,
\begin{equation*}
  \begin{split}
\EXP \{ Q_k (\hT_n) \}
&\le \CCk k+\CC  \sum_{n \ge \ell > n/2}  \frac{k }{ \ell} \, 
\sqrt\frac{ n }{ n-\ell+1} \\
&\le \CCk k+\CC \frac kn \sum_{j=1}^n  \sqrt\frac{n}{ j} \\
&\le \CC k. \\
  \end{split}
\end{equation*}
This proves \refT{TP1}, which by the argument at the beginning of the
section yields \refT{T1}.
\end{proof}

\section{Proof of \refL{L0} and \refT{T2}} \label{Sl2}

Denote the \lhs{} of \eqref{l2} by $\psint$. 
Since 
$
\sum_j X(j;\tn)e^{\ii jt} = \sumvtn e^{\ii t L_v}
$,
we have
\begin{equation}
  \label{b1}
\psint
=n\qww
\E\,\biggabs{\sumvtn e^{\ii t L_v}}^2
=n\qww
\E\sumvwtn e^{\ii t (L_v-L_w)}.
\end{equation}
Condition on $T_n$ and consider two vertices $v$ and $w$ in $\tn$. If
$\vw$ is the last common ancestor of $v$ and $w$, then $L_v-L_{\vw}$
and $L_w-L_{\vw}$ are (conditionally, given $\tn$) independent sums of
$d(v,\vw)$ and $d(w,\vw)$ copies of $\eta$,
respectively. Consequently, letting $\gfeta(t)\=\E e^{\ii t\eta}$ be the 
characteristic function of $\eta$,
\begin{equation*}
  \begin{split}
\E \bigpar{e^{\ii t (L_v-L_w)}\mid\tn}
&=
\E \bigpar{e^{\ii t (L_v-L_{\vw})}\mid\tn}
\E \bigpar{e^{-\ii t (L_w-L_{\vw})}\mid\tn}
\\&
=\gfeta(t)^{d(v,\vw)}\,\overline{\gfeta(t)}^{d(w,\vw)}.
  \end{split}
\end{equation*}
Hence, by \eqref{b1} and \eqref{hn},
\begin{equation*}
  \psint
= n\qww \E\sumvwtn\gfeta(t)^{d(v,\vw)}\overline{\gfeta(t)}^{d(w,\vw)}
=n\qww \hn\bigpar{\gfeta(t),\overline{\gfeta(t)}}
\end{equation*}
and \refT{Tgen2} yields
\begin{equation}
  \label{b2}
\psint \le \CCgenii n\qw|1-\gfeta(t)|\qww.
\end{equation}

Since $\E \eta=0$ and $\E \eta^2=\gssy<\infty$, 
we have
$\gfeta(t)=\exp\bigpar{-\tfrac12\gssy t^2+o(t^2)}$
for small $|t|$; moreover, since $\eta$ has span 1, $\gfeta(t)\neq1$ for
$0<|t|\le\pi$. It follows
that $\psi(t)\=(1-\gfeta(t))/t^2$ is a continuous non-zero function on
$[-\pi,\pi]$ (with  $\psi(0)\=\frac12\gssy$); hence, by compactness,
$|\psi(t)|\ge \cc$ for some $\cc>0$, and thus
\begin{equation*}
  |1-\gfeta(t)|\ge \ccx t^2,
\qquad |t|\le\pi.
\end{equation*}

It now follows from \eqref{b2}, and the obvious fact that
$\psint\le1$, that
\begin{equation*}
  (1+nt^4)\psint
\le
  1+nt^4\psint
\le
1+ \CCgenii \frac{t^4}{|1-\gfeta(t)|^2}
\le \CC.
\end{equation*}

This proves \refL{L0}, which as remarked in \refS{S:intro} implies
\refT{T2} by \cite[Remark 3.7]{SJ185}.

\section{Proof of \refT{Tgen2}} \label{Sgen2}

We use some further generating functions. 
Recall that $\cT$ is the (unconditioned) \GWt{}
with offspring distribution $\xi$, and define
\begin{align*}
  \Phi(z) &\= \E z^\xi, \\
  F(z) &\= \E z^{|\cT|}, \\
  G(z,x) &\= \E \Bigpar{z^{|\cT|}\sumvt x^{d(v,\oo)}}, \\
  H(z,x,y) &\= \E \Bigpar{z^{|\cT|}\sumvwt x^{d(v,\vw)}y^{d(w,\vw)}}
  =\sum_{n=1}^\infty \P(|\cT|=n)\hn(x,y)z^n.
\end{align*}
These functions are defined and analytic at least for $|z|,|x|,|y|<1$.

Let us condition on the degree $\doo$ of the root of $\cT$, recalling
that $\doo\eqd\xi$. 
If $\doo=\ell$, then $\cT$ has $\ell$ subtrees $\cT_1,\dots,\cT_\ell$
at the root $\oo$, and conditioned on $\doo=\ell$, these are independent
and with the same distribution as $\cT$;
we denote their roots (the neighbours of
$\oo$), by $\oo_1,\dots,\oo_\ell$.

Assume $\doo=\ell$, and let $|z|,|x|,|y|<1$.
First, $|\cT|=1+\sumil|\cT_i|$ and thus 
$z^{|\cT|}=z\prodil z^{|\cT_i|}$. Taking the expectation, we obtain, as
is well-known, first
\begin{equation*}
\E\bigpar{z^{|\cT|}\mid\doo=\ell}
=
z\E \prodil z^{|\cT_i|}
= z F(z)^\ell,
\end{equation*}
and then
\begin{equation*}
F(z)=\E\bigpar{z^{|\cT|}}
= z \suml \P(\xi=\ell)F(z)^\ell
=z\phifz.
\end{equation*}

Similarly, separating the cases
$v\in\cT_i$, $i=1,\dots,\ell$, and $v=\oo$, 
\begin{equation*}
  \sumvt x^{d(v,\oo)}
=
\sumil\sum_{v\in\cT_i} x^{d(v,\oo_i)+1} +1.
\end{equation*}
Hence,
\begin{align*}
\E\bigpar{z^{|\cT|}\sumvt x^{d(v,\oo)}\mid\doo=\ell}
&=
\E\sumil z z^{|\cT_i|}\sum_{v\in\cT_i} x^{d(v,\oo_i)+1} 
  \prod_{j\neq i} z^{|\cT_j|}
 +\E \Bigpar{z\prodil z^{|\cT_i|}}
\\
&= \ell zx G(z,x)F(z)^{\ell-1}+ z F(z)^\ell
\end{align*}
and
\begin{equation*}
  G(z,x)=\sum_{\ell=0}^\infty\P(\xi=\ell)\ell zxG(z,x)F(z)^{\ell-1}+F(z)
=zx\Phi'(F(z))G(z,x)+F(z)
\end{equation*}
which gives
\begin{equation}\label{yg}
  G(z,x)=\frac{F(z)}{1-zx\Phi'(F(z))}.
\end{equation}

Similarly,
\begin{align*}
\E\Bigpar{z^{|\cT|}&\sumvwt x^{d(v,\vw)}y^{d(w,\vw)}\mid\doo=\ell}
=
\\
&\E\sumil z z^{|\cT_i|}\sum_{v,w\in\cT_i} x^{d(v,\vw)}y^{d(w,\vw)}
  \prod_{j\neq i} z^{|\cT_j|}
\\&
\quad+
\E\sum_{i\neq j} z z^{|\cT_i|}\sum_{v\in\cT_i} x^{d(v,\oo_i)+1} 
  z^{|\cT_j|}\sum_{w\in\cT_j} y^{d(w,\oo_j)+1} 
  \prod_{k\neq i,j} z^{|\cT_k|}
\\&
\quad+
\E\sumil z z^{|\cT_i|}\sum_{v\in\cT_i} x^{d(v,\oo_i)+1} 
  \prod_{k\neq i} z^{|\cT_k|}
\\&
\quad+
\E\sumjl z z^{|\cT_j|}\sum_{w\in\cT_j} y^{d(w,\oo_j)+1} 
  \prod_{k\neq j} z^{|\cT_k|}
 +\E \Bigpar{z\prodil z^{|\cT_i|}}
\end{align*}
leading to
\begin{multline*}
  H(z,x,y)
=
z\Phi'(F(z))H(z,x,y)
+
zxy\Phi''(F(z))G(z,x)G(z,y)\\
+
zx\Phi'(F(z))G(z,x)
+
zy\Phi'(F(z))G(z,y)
+F(z)
\end{multline*}
which gives
\begin{multline}
\label{yh}
  H(z,x,y)
=\\
\frac{zxy\Phi''(F(z))G(z,x)G(z,y)
+
z\Phi'(F(z))\bigpar{xG(z,x)+ yG(z,y)}
+F(z)}
{1-z\Phi'(F(z))}
\end{multline}

Assume now for simplicity that $\xi$ has span 1. (The case when the
span is $d>1$ is treated similarly with the standard modification that
we have to give special treatment to neighbourhoods of the $d$:th unit
roots.)
Then, by \cite[Lemma A.2]{SJ167},
for some $\gd>0$ and $\gb\le\pi/4$, $F$ extends to an analytic function
in $\gdbd$ with $|F(z)|<1$ for $z\in\gdbd$ and
\begin{equation}
  \label{a5}
F(z)=1-\sqrt2\gs\qw\sqrt{1-z}+\oqq,
\qquad
\text{as $z\to1$ with } z\in\gdbd.
  \end{equation}
We will prove the following companion results.
\begin{lemma}\label{Lx}
  If\/ $\xi$ has span $1$, then there exists $\gb,\gd>0$ such that 
$F$ extends to an analytic function in $\gdbd$ and, for some 
$\cc\ccdef\cclxa,\cc\ccdef\ccj>0$,
if $x,z\in\gdbd$, then
\begin{align}\label{sofie}
  |1-z\Phi'(F(z))|&\ge \cclxa|1-z|\qq,
\\
  |1-xz\Phi'(F(z))|&\ge \ccj|1-x|.\label{julie}
\end{align}
Consequently, $G(z,x)$ and $H(z,x,y)$ extend to analytic functions
of $x,y,z\in\gdbd$, and, for 
all  $x,y,z\in\gdbd$,
\begin{align}
  |G(z,x)| &\le \CC|1-x|\qw, \label{erika}
\\
  |H(z,x,y)| &\le \CC|1-z|\qqw|1-x|\qw|1-y|\qw. \label{magnus}
\end{align}
\end{lemma}

Standard singularity analysis 
\cite[Lemma IX.2]{FS}
applied to \eqref{magnus} yields
\begin{align*}
|\P(|\cT|=n)\hn(x,y)|
&
\le \CC n \qqw|1-x|\qw|1-y|\qw,
\qquad
x,y\in\gdbd,
\end{align*}
which proves \refT{Tgen2} because, as is well known, a singularity
analysis of \eqref{a5} yields 
\begin{align*}
\P(|\cT|=n)\sim \cc n^{-3/2}.
\end{align*}
It thus remains only to prove \refL{Lx}.

\begin{proof}[Proof of \refL{Lx}]
Since $\E\xi^2<\infty$, $\Phi'$ and $\Phi''$ extend to continuous
functions on the closed unit disc with $\Phi'(1)=\E\xi=1$ and
$\Phi''(1)=\E\xi(\xi-1)=\gss$. Hence, \eqref{a5} yields, for $z\in\gdbd$,
\begin{equation*}
  \begin{split}
\Phi'(F(z))&=\Phi'(1)+\Phi''(1)\bigpar{F(z)-1}+o(|F(z)-1|)
\\&
=1-\sqrt2\gs\sqrt{1-z}+\oqq	
  \end{split}
  \end{equation*}
and
\begin{equation}
  \label{david}
z\Phi'(F(z))
=\Phi'(F(z))+O(|z-1|)
=1-\sqrt2\gs\sqrt{1-z}+\oqq.
  \end{equation}
Let $\be\=\set{z:|z-1|<\eps}$, and take $\gb<\pi/4$.
Since $z\in\ogdbdx$ entails $|\arg(1-z)|\le\pi/2+\gb$ and thus 
$|\arg\sqrt{1-z}|\le\pi/4+\gb/2$, it follows from \eqref{david} that,
for some small $\eps>0$, 
if $z\in\overline{\gdbde}$ with $z\neq1$,
then 
\eqref{sofie} holds, 
$\bigl|{z\Phi'(F(z))-1}\bigr|=O(\eps\qq)$,
\begin{equation}
  \label{emma}
  \begin{split}
\bigl|\arg\bigpar{z\Phi'(F(z))-1}\bigr|
&
>|\arg(-\sqrt{1-z})|-\gb/2
\\&
\ge
\pi-(\pi/4+\gb/2)-\gb/2
=3\pi/4-\gb,	
  \end{split}
\end{equation}
  and consequently, since $3\pi/4-\gb>\pi/2$,
if $\eps$ is small enough,
  \begin{equation}
	\label{samuel}
|z\Phi'(F(z))|<1.
  \end{equation}

Similarly, if $x\in\gdbd$, then $|\arg(1-x)|<\pi/2+\gb$ and
\begin{equation*}
  x\qw=\bigpar{1-(1-x)}\qw = 1+(1-x)+o(|1-x|),
\qquad x\to1,
\end{equation*}
so if $\eps>0$ is small enough,
then, for $x\in\gdbd\cap\be$,
\begin{equation}\label{jesper}
  |\arg(x\qw-1)|<\pi/2+2\gb.
\end{equation}
If we choose $\gb\le\pi/16$, it follows from \eqref{emma} and
\eqref{jesper} that the triangle with vertices in $1$, $x\qw$ and 
$z\Phi'(F(z))$ has an angle at least $\pi/4-3\gb\ge\pi/16$ at 1, 
and thus by elementary
trigonometry (the sine theorem), 
\begin{equation*}
|x\qw- z\Phi'(F(z))| \ge \cc |x\qw-1|,
\end{equation*}
and so \eqref{julie} holds, when $z,x\in\gdbd\cap\be$, 
provided $\gb,\gd,\eps$ are small enough.

It remains to treat the case when $x$ or $z$ does not belong to $\be$,
\ie, $|x-1|\ge\eps$ or $|z-1|\ge\eps$. We do this by compactness
arguments.

First, let 
\begin{align*}
  A&\=\set{z\Phi'(F(z)):z\in\gdbd\cap\be}
\\
\br&\=\set{x\qw:x\in\overline{\gdbr}\setminus\be,\, |x|\ge1/2}.
\end{align*}
Then $B\=\bigcap_{\rho>0}\br
\subset\set{\zeta:|\zeta|\ge1}\setminus\set1$,
and it follows from \eqref{samuel} that 
$\ba\cap B=\emptyset$.
Since $\ba$ and all $\br$ are compact, it follows that
$\ba\cap\br=\emptyset$ for some $\rho>0$, and thus, if
$z\in\gdbd\cap\be$ and $x\in\gdbr\setminus\be$ with $|x|\ge1/2$, then
$|x\qw-z\Phi'(F(z))|\ge \cc$ for some $\ccx>0$, which implies \eqref{julie}
for such $z$ and $x$. 
Moreover, if $z\in\gdbde$ and $|x|<1/2$, \eqref{samuel} shows that 
$|1-xz\Phi'(F(z))|\ge 1-|x|\ge1/2$, so \eqref{julie} then holds if 
$\ccj\le1/3$.

Finally, if $z\in\gdbd$, then
$|F(z)|<1$  \cite[Lemma A.2]{SJ167} as stated above,
and thus $|\Phi'(F(z))|<1$.
If $0<\gb_1<\gb$ and $0<\gd_1<\gd$, then
$\overline{\gdbdi}\subset\gdbd\cup\set1$, and thus by compactness
\begin{equation*}
  \ce\=\sup\bigset{|\Phi'(F(z))|:z\in\overline{\gdbdi}\setminus\be}<1.
\end{equation*}
Consequently, if $\gd_2\le\gd_1$ is small enough and $x,z\in\gdbdii$ with
$|z-1|\ge\eps$, then
\begin{equation*}
  |xz\Phi'(F(z))|\le(1+\gd_2)^2\ce<1.
\end{equation*}
Hence \eqref{julie} holds in this case too for some $\ccj>0$, and
similarly \eqref{sofie} holds for $z\in\gdbdii\setminus\be$.

This completes the proof of \eqref{sofie} and \eqref{julie}, for some new
$\gb,\gd>0$
(\viz, $\gb_1$ and $\min(\gd_2,\rho)$).
$G(z,x)$ now can be defined for all $x,z\in\gdbd$ by \eqref{yg}, 
and \eqref{erika}
holds by \eqref{julie}.
Similarly, 
$H(z,x,y)$ can be defined for all $x,y,z\in\gdbd$ by \eqref{yh}, 
and \eqref{magnus} holds by \eqref{sofie}, \eqref{erika}, and the fact
that $\Phi'$ and $\Phi''$ are bounded on the unit disc. (Recall that
$|F(z)|<1$ for $z\in\gdbd$.)

This completes the proof of \refL{Lx}, and thus of \refT{Tgen2} and of
all results in this paper.
\end{proof}

\newcommand\AAP{\emph{Adv. Appl. Probab.} }
\newcommand\JAP{\emph{J. Appl. Probab.} }
\newcommand\JAMS{\emph{J. \AMS} }
\newcommand\MAMS{\emph{Memoirs \AMS} }
\newcommand\PAMS{\emph{Proc. \AMS} }
\newcommand\TAMS{\emph{Trans. \AMS} }
\newcommand\AnnMS{\emph{Ann. Math. Statist.} }
\newcommand\AnnAP{\emph{Ann. Appl. Probab.} }
\newcommand\AnnPr{\emph{Ann. Probab.} }
\newcommand\CPC{\emph{Combin. Probab. Comput.} }
\newcommand\JMAA{\emph{J. Math. Anal. Appl.} }
\newcommand\RSA{\emph{Random Struct. Alg.} }
\newcommand\ZW{\emph{Z. Wahrsch. Verw. Gebiete} }
\newcommand\DMTCS{\jour{Discr. Math. Theor. Comput. Sci.} }

\newcommand\AMS{Amer. Math. Soc.}
\newcommand\Springer{Springer-Verlag}
\newcommand\Wiley{Wiley}

\newcommand\vol{\textbf}
\newcommand\jour{\emph}
\newcommand\book{\emph}
\newcommand\inbook{\emph}
\def\no#1#2,{\unskip#2, no. #1,} 
\newcommand\toappear{\unskip, to appear}

\newcommand\webcite[1]{\hfil\penalty0\texttt{\def~{\~{}}#1}\hfill\hfill}
\newcommand\webcitesvante{\webcite{http://www.math.uu.se/\~{}svante/papers/}}
\newcommand\arxiv[1]{\webcite{arXiv:#1}}

\def\nobibitem#1\par{}

\end{document}